\documentclass[reqno,11pt]{amsart}
\usepackage{amsmath,amsthm,amssymb,amsfonts}
\usepackage{graphicx}
\usepackage{amsmath,amsfonts,amssymb,amsthm,epsfig,color}

\newtheorem{theorem}{Theorem}[section]
\newtheorem{corollary}[theorem]{Corollary}
\newtheorem{lemma}[theorem]{Lemma}

\newtheorem{proposition}[theorem]{Proposition}

\newtheorem*{mt}{Main Theorem}
\newcommand{\R}{\mathbb{R}}

\numberwithin{equation}{section}

\def\N{\mathbb N}
\def\R{\mathbb R}

\def\a{\alpha}

\def\g{\gamma}

\def\bal{\begin{aligned}}
\def\eal{\end{aligned}}

\title[Stability theorems for GNS inequalities]{Stability theorems for
Gagliardo-Nirenberg-Sobolev inequalities:\\ a reduction principle to the radial case}

\author{B. Ruffini}
\address{Scuola Normale Superiore di Pisa, Piazza dei Cavalieri 4, 56126 Pisa, ITALY  }
\email{berardo.ruffini@sns.it}

\begin{document}

\begin{abstract}
In this paper we investigate the quantitative stability for Gagliardo-Nirenberg-Sobolev inequalities.
The main result is a reduction theorem, which states that, to solve the problem of the stability for
Gagliardo-Nirenberg-Sobolev inequalities,
one can consider only the class of radial decreasing functions. 
\end{abstract}

\maketitle

\section{Introduction}
\subsection{Background}
The sharp Gagliardo-Nirenberg-Sobolev inequality (briefly: GNS) in $\R^n$, with $n\geq2$, takes the form, 
for a suitable constant $G=G(n,p,s,q)>0$,
\begin{equation}\label{GNS}
G\|u\|_q\leq\|\nabla u\|_p^\theta\|u\|_s^{1-\theta}
\end{equation}
where the parameters $s$, $q$, $p$  satisfy
\begin{equation}\label{npqs}
\begin{aligned}
&1< p<n\text{,} \qquad \\
&1\leq s<q<p^\star\text{,} \qquad p^\star=\frac{np}{n-p}\qquad\text{and}\\
&\frac{\theta}{p^\star}+\frac{1-\theta}{s}=\frac{1}{q}.
\end{aligned}
\end{equation}
In \eqref{GNS}, and in what follows in this paper, the symbol $\|u\|_r$ denotes the $L^r$-norm of a function
$u$ on  $\R^n$
\[
 \|u\|_r^r=\int_{\R^n}|u|^r=\int|u|^r
\]
while if the set of integration  is a proper subset $E$ of $\R^n$, we  use the notation
$\|\cdot\|_{L^r(E)}$. Analogously, to lighten the notations, we shall omit the set
of integration when an integral is defined on the whole $\R^n$. The function $u$ is taken in $D^{p,s}(\R^n)$, that
is the closure of $C^\infty_c(\R^n)$ under
the norm $\|u\|_{D^{p,s}}=\|\nabla u\|_p+\|u\|_s$, anyhow we can always suppose, in this paper, to deal with
regular functions, so
that the gradient $\nabla u$ of a function $u$ is well defined. 

 Inequality \eqref{GNS} belongs to the class of the celebrated {\em Sobolev inequalities}, which are a
fundamental tool in several applications in the Calculus of Variation, 
and it 
 can be derived by
combining the interpolation inequality between the $L^s$, $L^q$ and $L^{p^\star}$ norms on $\R^n$ with the
classical Sobolev inequality
\begin{equation}
S(n,p)\|u\|_{p^\star}\leq\|\nabla u\|_p.
\end{equation}
\noindent Here $S(n,p)$ is the optimal Sobolev constant, namely
\begin{equation}\label{costantesobolev}
S(n,p)=\inf\Bigg\{\frac{\|\nabla u\|_p}{\|u\|_{p^\star}}:\quad u\in W^{1,p}(\R^n)\setminus{\{0\}} \Bigg\},
\end{equation}
being $W^{1,p}(\R^n)$ the Sobolev space of functions which are, together with their weak gradient, in $L^p(\R^n)$
(see \cite{Evans}). 
\noindent Explicit formulas for $S(n,p)$ and minimizers in \eqref{costantesobolev} are known since the work 
of Aubin \cite{Au} and Talenti \cite{Ta}. We stress that the same result {\it is not available} for the optimal
constants and functions in \eqref{GNS} with the exception of the one-parameter family of exponents
\begin{equation}\label{kp}
 p(s-1)=q(p-1)\,\,\,\,\text{if $q,s>p$}\qquad p(q-1)=s(p-1)\,\,\,\,\text{if $q,s<p$},
\end{equation}
 see \cite{DelPinoDolbeault,CENV}. What can be said in full generality, is that optimal functions in
\eqref{GNS} exist and, according to the P\'olya-Szeg\"{o} inequality (see for instance \cite{BZ}), are
non-negative, radially symmetric functions with decreasing profile (see \cite{BZ,ST}).
Furthermore, they are unique up to translations, rescaling and multiplication by (non-zero) constants.
 In this paper we deal with the problem of getting a quantitative stability version of inequality \eqref{GNS}. 
 To state rigorously this problem, we introduce the (GNS) \emph{deficit} $\delta(u) $ of a function $u\in
D^{p,s}(\R^n)$ as
 \begin{equation}\label{deficit}
 \delta(u)=\frac{\|\nabla u\|_p^\theta\|u\|_s^{1-\theta}}{G\|u\|_q}-1,
 \end{equation}
and notice that inequality \eqref{GNS} reads, in terms of $\delta(u)$, as
\begin{equation}\label{GNS1}
\delta(u)\ge0.
\end{equation} 
 Then, for a quantitative version of the GNS inequality, we mean an improvement of inequality \eqref{GNS1} of the
form
 \begin{equation}
 \delta(u)\ge \kappa_0\rm{dist}(u,M)^{\a_0},
 \end{equation}
where $\kappa_0,\a_0$ are positive constants not depending on $u$ and $\rm{dist}(\cdot,M)$ indicates an 
appropriate distance from $M$, the set of the  optimizers for \eqref{GNS}. The concept of distance we shall adopt
is
the following 
\begin{equation}\label{asimmetria}
\lambda(u)=\inf\left\{\frac{\|u-v\|_q^q}{\|u\|_q^q}:\text{$v$ is optimal for \eqref{GNS}},\quad
\|v\|_q=\|u\|_q\right\}
\end{equation}
and we will call it {\em asymmetry of $u$}.
Results in this direction have been recently obtained with some ad-hoc techniques valid for special classes of
parameters 
among those in \eqref{npqs}, see  \cite{CaF} and \cite{CaFL}. In particular, the parameters considered in
\cite{CaF} are contained in those introduced in \eqref{kp} (although the authors focus, because of their later
applications,
just on the particular case $p=2, q=6, s=4$), and the
knowledge of
minimizers is exploited in a crucial way. In \cite{CaFL}, the authors address a class of parameters, $p=s=2, q>2$,
for which the minimizers are not explicitly known and they follow a strategy developed by
Bianchi and Egnell in \cite{be}, which heavily relies on the Hilbertian structure corresponding to $p=2$ and seems
complicate to generalize. Thus the above techniques seem adaptable to prove the stability for the GNS inequalities
only for a particular class of parameters. In this paper, thanks to a general symmetrization technique introduced
by Cianchi, Fusco, Maggi and Pratelli in \cite{CFMP}, we are able to prove a reduction principle which is valid
for the \emph{whole class} of parameters \eqref{npqs}. Namely we reduce the
problem to that of showing the stability just for radial symmetric functions, reducing the complexity of the task
in its generality from a $n-$dimensional to a $1-$dimensional problem. Although this does not solves completely
the problem, it offers a more simple way to attack it. 

\subsection{Main result and plan of the paper}

Consider the class of parameters \eqref{npqs}. The main result we shall prove is the following

\begin{mt}\label{principale}
Consider the functionals $\delta(\cdot)$ and $\lambda(\cdot)$ defined in \eqref{deficit} and
\eqref{asimmetria} respectively. Suppose that there exist two positive constants $k_0$ and $\a_0$ such that the
stability inequality 
\begin{equation}\label{gnsq}
\delta(u)\ge \kappa_0\lambda(u)^{\a_0}
\end{equation}
holds for any radial non-increasing function $u\in D^{p,s}(\R^n)$. Then there exist two positive constants $k_1$
and
$\a_1$ such that the inequality
\begin{equation}
\delta(u)\ge \kappa_1\lambda(u)^{\a_1}
\end{equation}
holds true for any function in $D^{p,s}(\R^n)$.
\end{mt}
A first result we shall use, in the direction of proving the Main Theorem, is a sort of {\it continuity at
$0$} of the asymmetry
$\lambda$ with respect to the deficit $\delta$. Namely we prove, in Corollary
\ref{continuity} of Section \ref{softstability}, that given a sequence of functions $(u_h)_h$ such that
$\delta(u_h)$ converges to $0$ as $h$ goes to infinity, then  $\lambda(u_h)$ converges to $0$ as well. This
result is accomplished by
means of the compactness Theorem \ref{compattezza}, where we prove that a sequence of functions whose deficits
are infinitesimal, up to be (suitably) rescaled and translated, is compact in $L^q(\R^n)$. 
\noindent
After Corollary \ref{continuity} is settled, we pass to the proof of the Main Theorem. Its proof is
given in Section \ref{n-sym} and Section \ref{reduc}, each of them devoted to obtain a simplification of the
class of 
functions we deal with. More precisely we make use of a further reduction step, aimed
to prove that if the stability inequality \eqref{gnsq}
holds true for radial decreasing functions, it also holds  for $n$-symmetric functions, that is functions which are
symmetric with respect to $n$ orthogonal hyperplanes. Namely we prove, in Section 
\ref{n-sym}, that if  there exist positive
constants $\kappa_0$ and $\alpha_0$ 
such that for any radial decreasing function  $u\in D^{p,s}(\R^n)$ inequality \eqref{gnsq} holds true, then there
exist positive constants $\kappa_1$ and $\alpha_1$ depending on $n,p,q$ and $s$ such that for any $n$-symmetric
function $u\in D^{p,s}(\R^n)$ we have
\begin{equation}\label{formulagenerale}
\delta(u)\ge \kappa_1\lambda(u)^{\alpha_1}.
\end{equation}
\noindent
Eventually, in Section \ref{reduc}, we prove that to get the stability for GNS inequality, it is not restrictive 
to consider only $n$-symmetric functions. To do this, we show the existence of two positive constants $\kappa_2$
and
$\alpha_2$ such that for every non-negative function $u\in D^{p,s}(\R^n)$ there exists an $n$-symmetric function
$\bar{u}$ such that the following reduction inequalities hold true
\begin{equation}\label{riduzione}
\lambda(u)\leq \kappa_2\lambda(\bar{u})^{\alpha_2},\qquad   \delta(\bar{u})\leq \kappa_2\delta(u)^{\alpha_2}.
\end{equation}
It is then easy to see that combining \eqref{formulagenerale} with the reduction inequalities \eqref{riduzione} 
we get the claim of the Main Theorem.

\section{Continuity of $\lambda$ with respect to $\delta$ via a compactness theorem}\label{softstability}

We devote this section to the proof of the following theorem.

\begin{theorem}\label{compattezza}
Let $(u_h)_h$ be a sequence in $D^{p,s}(\R^n)$ such that $\delta(u_h)$ converges to $0$ as $h\to\infty$. 
Then there exist $(\lambda_h)_h\subset(0,+\infty)$ and  $(x_h)_h\subset\R^n$ such that the rescaled sequence 
\begin{equation}\label{1}
w_h(x)=\tau_{\lambda_h}u_h(x-x_h)=\lambda_h^{n/q}u_h(\lambda_h (x-x_h)) 
\end{equation}
satisfies:
\begin{itemize}
\item[(i)] $\delta(w_h)=\delta(u_h)$;$\quad \lambda(w_h)=\lambda(u_h)$;
\item[(ii)] $\|w_h\|_q=\|u_h\|_q \quad{\rm for }\,\, h\in \mathbb{N}$;
\item[(iii)] there exist constants $ C_0,C_1>0$,  depending only on $n,p,q,s$,  such that 
\[
 \frac{1}{C_0}\leq\|\nabla w_h\|_p\leq C_0,\quad \frac{1}{C_1}\leq\|w_h\|_s\leq C_1;\]
\item[(iv)] $w_h\to w\quad\text{strongly in $L^q(\R^n)$ as $h\to\infty$ with $w\in D^{p,s}(\R^n)$}$.
\end{itemize}
\end{theorem}
\noindent  In what follows we will always implicitly suppose that the parameters $n,p,s,q$ and $\theta$ satisfy
conditions \eqref{npqs}. We begin our analysis considering the following functionals
\begin{equation}\label{FeG1}
G(u)=\|\nabla u\|_p^\theta\|u\|_s^{1-\theta},\qquad F(u)=\int|\nabla u|^p+\int|u|^s
\end{equation}
\noindent defined for $u\in D^{p,s}(\R^n)$. Given $m>0$ we define
the functions
\begin{equation}\label{phiepsi}
\psi(m)=\inf\{G(u):\|u\|_q^q=m\},\qquad \varphi(m)=\inf\{F(u):\|u\|_q^q=m\}.
\end{equation}
\begin{lemma}\label{FeGcollegate}
There exists $\eta_0=\eta_0(n,p,q,s)>0$ with the following property. For any $u\in D^{p,s}(\R^n)$ there 
exists $\lambda>0$ such that, if $\tau_\lambda(u)=\lambda^{n/q}u(\lambda x)$, then
\begin{equation}
F(\tau_\lambda u)=\eta_0G(u)^k\qquad\text{where}\qquad k=q\cdot\frac{np+ps-ns}{np+pq-ns}<q.
\end{equation}
\end{lemma}
\begin{proof}
For $u\in D^{p,s}(\R^n)$ we have
\begin{equation}\label{cambi}
\|\tau_\lambda u\|_q^q=\|u\|_q^q,\quad \|\tau_\lambda u\|_s^s=\lambda^{-n+ns/q}\|u\|_s^s\quad\text{and}\quad 
\|\nabla \tau_\lambda u\|_p^p=\lambda^{-n+p+\frac{np}{q}}\|\nabla u\|_p^p;
\end{equation}
hence
\[
F(\tau_\lambda u)=\lambda^aA+\lambda^bB=f(\lambda)
\]
where $A=\|\nabla u\|_p^p$, $B=\|u\|_s^s$, $a=-n+p+np/q$ and $b=-n+ns/q$. The function $f$ attains its minimum at
\[
\lambda_m=\Big(-\frac{b}{a}\Big)^{\frac{1}{a-b}}\Big(\frac{B}{A}\Big)^{\frac{1}{a-b}}
\]
with the value
\[
f(\lambda_m)=\eta_0\Big(A^{\theta/p}\Big)^{q\nu}\Big(B^{(1-\theta)/s}\Big)^{q\nu}
\]
where $\eta_0=\eta_0(n,p,s,q)$ and $\nu=\frac{np+ps-ns}{np+pq-ns}$, that is the claim of the lemma.
\end{proof}
\begin{lemma}\label{supadd}
There exists $\alpha=\alpha(n,p,q,s)\in(0,1)$ such that 
\[
\varphi(m)=m^\alpha\varphi(1)\qquad \text{ for all}\,\, m>0.
\]
In particular $\varphi$ is strictly super-additive in $(0,1)$.
\end{lemma}
\begin{proof}
Let $u\in D^{p,s}(\R^n)$ be such that $\|u\|_q^q=m$. Let $v=u/m^{1/q}$ and set $\tau_\lambda v(x)
=\lambda^{n/q}v(\lambda x)$. Setting $\lambda=m^{\frac{p-s}{np+pq-ns}}$ we get, after some calculation analogous to
those in Lemma \ref{FeGcollegate},
\[
F(\tau_\lambda v)=m^{-\alpha}F(u)
\]
where $\alpha=\frac{np+ps-ns}{np+pq-ns}$. If we now consider a minimizing sequence $(u_h)_h$ for $\varphi(m)$ 
such that $\|u_h\|_q^q=m$ and as above $v_h=u_h/m^{1/q}$, we obtain
\[
\varphi(m)= \lim_{h\to\infty} F(u_h)=m^\alpha\lim_{h\to\infty}F(\tau_\lambda v_h)\geq m^\alpha\varphi(1).
\]
The opposite inequality can be proved with an analogous argument considering a minimizing sequence $(v_h)_h$
for $\varphi(1)$, and setting $u_h=m^{1/q}v_h$.
\end{proof}

\noindent
We pass now to the proof of Theorem \ref{compattezza}.
\begin{proof}[Proof of Theorem \ref{compattezza}]
\noindent Let $\varphi$ be the function defined in \eqref{phiepsi}. We recall the Lions's
Concentration-Compactness 
Theorem (see \cite{llcc1} or \cite[ Theorem 4.3 and Theorem 4.8]{S}): given a non-negative sequence $(\rho_h)_h$ in
$L^1(\R^n)$ with fixed $L^1$ norm, say $1$, there is a subsequence $(\rho_{h_k})_k$ which satisfies one of the
following properties:
\begin{itemize}
\item[(1)] (concentration) there exists a sequence $(y_k)\subset\R^n$ such that for every $\varepsilon>0$  there 
exists $ R\in (0,\infty)$ such that $\int_{B_R(y_k)}\rho_{h_k}\geq1-\varepsilon\qquad\text{for every $k\in\N$}$;
\item[(2)](vanishing) $\lim_{k\to\infty}\sup_{y\in\R^n}\int_{B_R(y)}\rho_{h_k}=0 \quad{\rm for}\quad
0<R<\infty$;
\item[(3)](dichotomy) there exist $\alpha\in(0,1)$ and two sequences $R_h\to+\infty$ and $(y_h)_h\subset\R^n$ such
that 
$\int_{B_{R_{h_k}}(y_{h_k})}\rho_{h_k}\to\alpha$ and $\int_{\R^n\setminus
B_{2R_{h_k}}(y_{h_k})}\rho_{h_k}\to1-\alpha$.

\end{itemize}
\noindent Let $u_h\in D^{p,s}(\R^n)$ be such that $\delta(u_h)\to0$. We can suppose that $\|u_h\|_q=1$. Consider 
$w_h(x-x_h)=\tau_{\lambda_h}u_h(x-x_h)$ where $\lambda_h>0$ is defined, for every $h$, as in Lemma
\ref{FeGcollegate}. Thanks to formulas \eqref{cambi}, each function $w_h$ satisfies  statements $(i)$ and $(ii)$ of
the theorem. Moreover, Lemma \ref{FeGcollegate} provides us two positive constants $\eta_0$ and $k$ such that
$F(w_h)=\eta_0G(u_h)^k$. Since $\delta(w_h)+1=G(w_h)/G$, and $\delta(w_h)$ tends to $0$, it follows that the
sequence $F(w_h)\to G^k=\varphi(1)$ as $h$ tends to $\infty$, where $\varphi$ is defined in \eqref{phiepsi}. In
particular the sequence $(w_h)_h$ must satisfy statement $(iii)$ of the theorem. In order to prove point {\em
(iv)}, we apply the Concentration-Compactness Theorem to the sequence $(|w_h|^q)_h$ aiming to exclude cases $(2)$
and $(3)$.

If the sequence vanishes, by H\"older inequality we would get vanishing also for the sequence $(|w_h|^s)_h$, 
since $s<q$. It is not difficult to see that these conditions, together to the equiboundedness of $(w_h)_h$ in
$D^{p,s}(R^n)$, guarantee that $w_h\to0$ strongly in $L^q$ as $h\to\infty$ (see for istance \cite{llc2} Lemma
$I.1$). Since $\|w_h\|_q=\|u_h\|_q=1$, we would get a contradiction. So we can exclude case $(2)$.

The dichotomy case is more complicate and requires a longer analysis. Suppose to have dichotomy for the sequence 
$(|w_h|^q)_h$. Then there exist $\alpha\in(0,1)$ and a sequence of positive numbers $R_h\to\infty$ as $h\to\infty$
such that
\[
\int_{B_{R_h}}|w_h|^q\to\alpha;\quad \int_{B_{2R_h}^c}|w_h|^q\to 1-\alpha;\quad \int_{B_{2R_{h}}\setminus 
B_{R_h}}|w_h|^q\to0.
\] 

\noindent Let $f\in C^1_c(B(0,2);[0,1])$ such that $f=1$ on $B(0,1)$ and consider 
$f_h(x)=f(x/R_h)\in C^1_c(B(0,2R_h);[0,1])$. Choose also $f$ such that $|\nabla f_h|\le C/R_h$ for some $C>0$. Then
we
have
\begin{equation*}
\begin{aligned}
F(w_h)&=\int|\nabla w_h|^p +\int|w_h|^s \\
	&=\int_{B_{R_h}}|\nabla (f_hw_h)|^p+\int_{B^{c}_{2R_h}}|\nabla ((1-f_h)w_h)|^p+\int_{B_{2R_h}\setminus
B_{R_h}}|\nabla w_h|^p \\
	&\quad+\int_{B_{R_h}}|f_hw_h |^s+\int_{B^{c}_{2R_h}}|(1-f_h)w_h |^s+\int_{B_{2R_h}\setminus B_{R_h}}|w_h|^s\\
	&=\int|\nabla (f_hw_h)|^p + \int|\nabla [(1-f_h)w_h]|^p \\
	&\quad+\int_{B_{2R_h}\setminus B_{R_h}}\Big[|\nabla w_h|^p-|\nabla (f_hw_h)|^p-|\nabla [(1-f_h)w_h]|^p\Big]
\\
	&\quad+\int\Big[|f_hw_h|^s+|(1-f_h)w_h|^s\Big] \\
	&\quad+\int_{B_{2R_h}\setminus B_{R_h}}\Big[|w_h|^s-f_h^s |w_h|^s-(1-f_h)^s |w_h|^s\Big].
\end{aligned}
\end{equation*}
\noindent Since $f_h$ takes values in $[0,1]$, the last integral is non-negative. Neglecting this quantity 
and rearranging the terms, we get
\begin{equation}
\begin{aligned}
F(w_h)
	&\geq F(f_hw_h) + F((1-f_h)w_h)- \epsilon(h)\\
	&\geq\varphi(\|f_hw_h\|^{q}_{q})+\varphi(\|(1-f_h)w_h\|^{q}_{q})-\epsilon(h),
\end{aligned}
\end{equation}
where 
\[
\epsilon(h)=\int_{B_{2R_h}\setminus B_{R_h}}\Big[|\nabla (f_hw_h)|^p+|\nabla [(1-f_h)w_h]|^p-|\nabla
w_h|^p\Big]\text{.}
\]
\noindent We claim that the error $\epsilon(h)$  is estimated from above by a quantity which converges to 
$0$ as $h\to\infty$. Indeed 
\begin{equation}\label{err}
\begin{aligned}
\epsilon(h)   &=\int_{B_{2R_h}\setminus B_{R_h}}\Big[|\nabla (f_hw_h)|^p+|\nabla [(1-f_h)w_h)]|^p-|\nabla
w_h|^p\Big] \\
	&= \int_{B_{2R_h}\setminus B_{R_h}}\Big[ |f_h\nabla w_h+w_h \nabla f_h|^p+|(1-f_h)\nabla w_h -w_h \nabla
f_h|^p-|\nabla w_h|^p\Big] \\
	&\leq \int_{B_{2R_h}\setminus B_{R_h}}\Big[\Big(|f_h\nabla w_h+w_h \nabla f_h| +|(1-f_h)\nabla w_h -w_h
\nabla f_h|\Big)^p-|\nabla w_h|^p\Big] \\
	&\leq \int_{B_{2R_h}\setminus B_{R_h}}\Big[\Big(|\nabla w_h|+2|w_h||\nabla f_h|\Big)^p-|\nabla w_h|^p\Big] 
    \\
	&\leq \int_{B_{2R_h}\setminus B_{R_h}}\Big[ \varepsilon C_p |\nabla w_h|^p+C_\varepsilon|\nabla
f_h|^p|w_h|^p\Big] \\
	&\leq\varepsilon C_p \sup_{h\in \N}\int|\nabla w_h|^p +C_\varepsilon \int_{B_{2R_h}\setminus
B_{R_h}}|\nabla f_h|^p|w_h|^p.
\end{aligned}
\end{equation}
\noindent where the first inequality is due to the super additivity of the map $t\mapsto t^p$ on $\R_+$, the 
second to the triangle inequality and the third one is  the Young inequality of (suitable) parameters
$\varepsilon>0$ and $C_\varepsilon$. We need to estimate the quantity
\[
g(h)=\int_{B_{2R_h}\setminus B_{R_h}}|\nabla f_h|^p|w_h|^p.
\]
If $p>s$ then interpolating the $L^p$ norm of the $w_h$'s between the $L^s$ norm and the $L^{p^\star}$ norm and 
recalling that $|\nabla f_h|\le C/R_h$,  we get
\[
\int_{B_{2R_h}\setminus B_{R_h}}|\nabla
f_h|^p|w_h|^p\le\frac{C^p}{R_h^p}\|w_h\|_{p^\star}^{p\theta}\|w_h\|_s^{p(1-\theta)}.
\]
Since we already know that $w_h$ satisfies statement $(iii)$ of the theorem, we get that $g(h)\to0$ as
$h\to\infty$. If $s\ge p$ we divide $g(h)$ into two terms:
\[
g(h)=\int_{(B_{2R_h}\setminus B_{R_h})\cap \{w_h\ge1\}}|\nabla f_h|^p|w_h|^p+\int_{(B_{2R_h}\setminus
B_{R_h})\cap\{w_h<1\}}|\nabla f_h|^p|w_h|^p=g_1(h)+g_2(h).
\]
Since $q>p$ we have
\[
g_1(h)\le \frac{C^p}{R_h^p}\int_{B_{2R_h}\setminus B_{R_h}}|w_h|^q\le\frac{C^p}{R_h^p}
\]
and so $g_1(h)\to0$ as $h\to\infty$. Moreover, by H\"older inequality of parameter $p^\star/p$, we get
\[
g_2(h)\le \frac{C^p}{R_h^p}\left(\int_{(B_{2R_h}\setminus B_{R_h})\cap\{w_h<1\}}|w_h|^{p^\star}\right)^{p/p^\star}
|(B_{2R_h}\setminus B_{R_h})|^{\frac{1}{(p/p^\star)'}}
\]
where 
\[
(p/p^\star)'=\frac{p/p^\star}{-1+p/p^\star}=n/p.
\] 
Since $p^\star>q$, we obtain
\[
g_2(h)\le (\omega_n(2^n-1))^{p/n}C^p \left(\int_{(B_{2R_h}\setminus B_{R_h})\cap\{w_h<1\}}|w_h|^{q}\right)^{p/p^\star}
\]
where $\omega_n$ is the measure of the unit ball of $\R^n$. So also $g_2$ converges to $0$ as $h\to\infty$. 
Thus, passing to the limit in \eqref{err}, first in $h\to\infty$ and then in $\varepsilon\to0$ we obtain that
$\epsilon(h)$ can be estimated from above by an infinitesimal quantity. Since $\|w_h f_h\|_{q}^q$ and
$\|(1-f_h)w_h\|_{q}^q$ converge respectively to $\lambda$ and 
$1-\lambda$, we can conclude thanks to  Lemma \ref{supadd} that
\[
\varphi(1)\ge\varphi(\lambda)+\varphi(1-\lambda)>\varphi(1)
\]
obtaining a contradiction.

\noindent So we can exclude also the dichotomy phenomenon. Since $w_h$ is equibounded in $L^q(\R^n)$, we can
consider
its weak-$L^q$ limit $w$. This is also a strong limit in $L^q$. Indeed by concentration, up to translations and
since $q>1$, we have
\[
1-\varepsilon\le \lim_{h\to\infty}
\int_{B_R}|w_h|^q=\int_{B_R}|w|^q\leq\int|w|^q\leq\liminf_{h\to\infty}\int|w_h|^q=1.
\]
This concludes the proof of the theorem.
\end{proof}

\begin{corollary}\label{continuity}
Consider a sequence $(u_h)_h\subset D^{p,s}(\R^n)$ such that $\delta(u_h)\to\infty$ as $h\to0$. Then also 
$\lambda(u_h)\to0$ for $h\to\infty$.
\end{corollary}
\begin{proof}
As a consequence of Theorem \ref{compattezza}, up to pass to a subsequence and to consider a rescaling of the form
$\tau_\lambda u_h(x)=\lambda^{n/q}u(\lambda x)$,  we can suppose that $u_h\to u$ strongly in $L^q(\R^n)$. 
Since the map $u\mapsto\lambda(u)$ is strongly continuous in $L^q(\R^n)$,  $\lambda(u_h)$ converges to
$\lambda(u)$. Furthermore, by the semicontinuity of the deficit (with respect to the $L^q$-convergence),
$\delta(u)=0$. Hence $u$ is optimal for \eqref{GNS} and $\lambda(u)=0$.
\end{proof}

\section{Reduction to $n$-symmetric functions}\label{n-sym}
In this section we prove that inequality \eqref{formulagenerale} holds true for $n$-symmetric functions 
(recall that a function is $k$-symmetric in $\R^n$, $n\ge k$, if it is symmetric with respect to $k$ mutually
orthogonal hyperplanes). 
We begin with a brief overview of the strategy we want to adopt. Given a function $u\in D^{p,s}(\R^n)$, the 
natural radial symmetric function to look at is its spherical rearrangement $u^\star$ (see \cite[Chapter
$3$]{liebloss}). Suppose that inequality \eqref{gnsq} holds true for
radial symmetric decreasing functions (and so for $u^\star$). Then by the triangle inequality we get
\[
\lambda(u)^{1/q}\leq \|u-u^\star\|_q +\lambda(u^\star)^{1/q}\le \|u-u^\star\|_q+\kappa_0^{1/q}\delta(u^\star)^{\a_0/q}.
\]
\noindent
In the direction of the proof of \eqref{formulagenerale}, we notice that by the P\'{o}lya-Szeg\"{o} inequality
(see for instance \cite{BZ}) we have that $\delta(u^\star)\leq\delta(u)$. 
But it is not clear if we can estimate the $L^q$ distance between $u$ and $u^\star$ in terms of $\delta(u)$. Indeed
this turns out to be true only if a function is already $n$-symmetric. We shall show, in Lemma \ref{lemma}, that
$\delta(u)$ estimates from above the P\'olya-Szeg\"o deficit, defined as
\begin{equation}\label{dps}
\delta_{PS}(u)=\frac{\|\nabla u\|_p-\|\nabla u^\star\|_{p}}{\|\nabla u^\star\|_{p}},
\end{equation}
and then, in Lemma \ref{controlloconps}, we  prove an estimate of the $L^{p^\star}$ distance between $u$ 
and $u^\star$ in terms of the $L^p$ distance between $|\nabla u|$ and $|\nabla u^\star|$. 
\begin{lemma}\label{lemma}
There exist two positive constants $\delta_0$ and $C_0$ such that for every $u\in D^{p,s}(\R^n)$ such that
$\|u\|_q=1$, with $\delta(u)\le\delta_0$, up to the rescaling \eqref{1}, we have
\begin{equation}\label{controllops}
\delta_{PS}(u)\leq C\delta(u)^{1/\theta}
\end{equation}
where $\theta\in(0,1)$ is the parameter introduced in \eqref{npqs}.
\end{lemma}
\begin{proof}
By Theorem \ref{compattezza}, up to rescaling it, we can suppose that $u$ satisfies properties $(i)-(iii)$ in
\eqref{1}. If we choose $\delta(u)\leq 1/G$, we obtain:
\begin{equation}\label{00}
G\leq\|\nabla u^\star\|_p^\theta\|u^\star\|_s^{1-\theta}\leq\|\nabla u\|_p^\theta\|u\|_s^{1-\theta}\leq 1+G.
\end{equation}
Then,
\begin{equation}\label{15}
\begin{aligned}
G\delta(u)&=\Big(\|\nabla u\|_p^\theta-\|\nabla u^\star\|_p^\theta\Big)\|u\|_s^{1-\theta}+G\delta(u^\star)\\
	&\geq C_1^{\theta-1}\Big(\|\nabla u\|_p^\theta-\|\nabla u^\star\|_p^\theta\Big)+G\delta(u^\star)
\end{aligned}
\end{equation}
where we used the fact that $\|u\|_q=1$ (statement $(ii)$) and $\|u\|_s\geq C_1^{-1}$ (statement $(iii)$). By \eqref{00} there exists a positive constant $c$ such that,
\[
\|\nabla u\|_p^\theta-\|\nabla u^\star\|_p^\theta\geq c\Big(\|\nabla u\|_p-\|\nabla u^\star\|_p\Big)^\theta.
\]
Now the conclusion follows from \eqref{15} and  \eqref{dps}, with $C=c^{1/\theta}$.
\end{proof}
To obtain the desired estimate of the $L^p$ distance between $u$ and $u^\star$ we shall use  the following result,
whose proof can be found in 
\cite[Theorem 3]{CFMP}.
\begin{lemma}\label{controlloconps}
Let $n\geq2$, $1<p<n$ and $z=\max\{p,2\}$. Then there exists a positive constant $C$ such that
\begin{equation}\label{mino}
\int|u-u^\star|^{p^\star}\leq C\Bigg(\int |u|^{p^\star}\Bigg)^{\frac{p}{n}}\Bigg(\int|\nabla
u^\star|^p\Bigg)^{\frac{z-1}{z}}\Bigg(\int|\nabla u|^p-\int|\nabla u^\star|^p\Bigg)^{\frac{1}{z}}
\end{equation}
holds for every non-negative $u\in W^{1,p}(\R^n)$ which is symmetric with respect to the coordinate hyperplanes.
\end{lemma}

 We are now able to proceed with the proof of \eqref{formulagenerale} when $u$ belongs to the class of
$n$-symmetric functions.

\begin{theorem}\label{t1}
Suppose that there exist positive constants $\kappa_0$ and $\alpha_0$ such that for any radial decreasing function 
$u\in D^{p,s}(\R^n)$ inequality \eqref{gnsq} holds true. Then, there exist positive constants $\kappa_1$ and
$\alpha_1$ depending on $n,p,q$ and $s$ such that for any $n$-symmetric function $u\in D^{p,s}(\R^n)$ we have
\begin{equation}\label{formulagenerale1}
\delta(u)\ge \kappa_1\lambda(u)^{\alpha_1}.
\end{equation}
\end{theorem}
\begin{proof}

Since $\lambda(u)\le 2^{q-1}$, if $\delta>0$ and $\delta(u)\geq \min\{\delta,1/G\}$, then we have $\lambda(u)\le
(2^{q-1}/\delta)\delta(u)$ 
and so \eqref{formulagenerale} holds true with $\kappa_1=2^{q-1}/\delta$ and $\a_1=1$. Hence we may assume  that
$\delta(u)\le\min\{\delta,1/G\}$ for a suitably small $\delta$.
 Moreover, by Theorem \ref{compattezza} we can suppose $\|u\|_q=1$ and $\|u\|_s\in[1/C_0,C_0]$ where $C_0\geq1$ is
a constant not depending on $u$. Keeping in mind these remarks, we divide the proof into two steps:

\noindent
\underline{Step 1}: We assume that $u\geq0$.
\noindent
In this case we have, by an interpolation inequality, Lemma \ref{controlloconps}, and Theorem \ref{compattezza},
\begin{equation}
\begin{aligned}\label{z}
\int|u-u^\star|^q &\leq\left(\int|u-u^\star|^s\right)^{(1-\theta)q/s}
\left(\int|u-u^\star|^{p^\star}\right)^{\theta q/p^\star} \\
		&\le C^{\frac{\theta q}{p^\star}}2^{\frac{(1-\theta)q(s-1)}{s}}\left(\int u^s\right) 
		\Bigg(\int|u|^{p^\star}\Bigg)^{\frac{p\theta q}{n
p^\star}}\Bigg(\int|\nabla u^\star|^p\Bigg)^{\frac{\theta q(z-1)}{p^\star
z}}\Bigg(\int|\nabla u|^p-\int|\nabla u^\star|^p\Bigg)^{\frac{\theta q}{p^\star z}} \\
		&\leq C(n,p,q,s)
		\left(\int|\nabla u|^p-\int|\nabla u^\star|^p\right)^{\g}.
\end{aligned}
\end{equation}
where $C(n,p,q,s)$ and $\g$ are suitable positive constants depending on $n,p,s$ and $q$. Notice that we used the
boundedness of the $L^p$ norm of $\nabla u$ and of the $L^s$ norm of $u$ guaranteed by Theorem \ref{compattezza}
(up
to choose $\delta$ small enough). Moreover we exploited Lemma \ref{controlloconps} and thus the assumption that
$u\ge0$. If we suppose $\delta(u)\leq\delta\leq G$, again by Theorem \ref{compattezza} we get
\[
G\leq C_2\|\nabla u^\star\|^{\theta}_{p}\leq C_2\|\nabla u\|^{\theta}_{p}=C_3G(\delta(u)+1)\leq C_3G(1+G),
\]
for suitable positive constants $C_2$ and $C_3$. Hence there exists $C_4>0$ such that
\[
\|\nabla u\|_p^p-\|\nabla u^\star\|_p^p\leq C_4(\|\nabla u\|_p-\|\nabla u^\star\|_p).
\]
\noindent By the triangle inequality, estimates  \eqref{controllops} and \eqref{z}, the P\'{o}lya-Szeg\"{o} 
inequality and the assumption on radial functions that we have as hypothesis, we can find constants $C_5$, $C_6$
and $C_7$ not depending on $u$ such that
\[
\begin{aligned}
\lambda(u)^q       &\leq2^{q-1}\left( \lambda(u^\star)+\|u-u^\star\|_q^q\right) \\
		&\leq C_5\left[\delta(u)^{\alpha_0}+\Bigg(\int|\nabla u|^p-\int|\nabla u^\star|^p\Bigg)^{\g}\right]
\\
		&\leq C_6\left[\delta(u)^{\alpha_0}+\delta(u)^{\g/\theta}]\right] \leq C_7\delta(u)^\xi,
\end{aligned}
\]
where $\xi=\min\{\a_0,\g/\theta\}$.

\underline{Step 2}: $u$ changes sign.
In this case consider the positive and the negative part of $u$: $u^+=u\chi_{\{u>0\}}$ and $u^-=-u\chi_{\{u<0\}}$ 
(where $\chi_A$ denotes the characteristic function of the set $A$).
By Lemma \ref{FeGcollegate} we are provided of a positive constant $\lambda$ such that
\[
\eta_0G(u)^\kappa=F(\tau_\lambda u).
\]
Moreover we have that
\[
F(\tau_\lambda u^\pm)\ge \inf_\mu F(\tau_\mu u^\pm)=\eta_0G(u^\pm)^\kappa.
\]
\noindent where $F$ and $G$ are defined in \eqref{FeG1}. So we get
\[
\begin{aligned}
G^\kappa(\delta(u)+1)^\kappa &=\big(\|\nabla u\|^{\theta}_{p}\|u\|^{1-\theta}_{s}\big)^\kappa
		   =\frac{1}{\eta_0}\Bigg(\int|\nabla \tau_\lambda u|^p+\int|\tau_\lambda u|^s\Bigg) \\
		&=\frac{1}{\eta_0}\Bigg(\int|\nabla \tau_\lambda u^{+}+\nabla \tau_\lambda u^{-}|^p+|\tau_\lambda
u^{+}+\tau_\lambda u^{-}|^s\Bigg) \\
		&=\frac{1}{\eta_0}\Bigg(\int|\nabla \tau_\lambda u^{+}|^p+\int|\nabla \tau_\lambda
u^{-}|^p+\int|\tau_\lambda u^{+}|^s+\int|\tau_\lambda u^{-}|^s\Bigg) \\
		&\ge\big(\|\nabla u^{+}\|^{\theta}_{p}\|u^{+}\|^{1-\theta}_{s}\big)^\kappa
+\big(\|\nabla u^{-}\|^{\theta}_{p}\|u^{-}\|^{1-\theta}_{s}\big)^\kappa \\
		&\geq G^\kappa\Big[\|u^{+}\|^{\kappa}_{q} + \|u^{-}\|^{^\kappa}_{q}\Big]=
G^\kappa\Big[\|u^+\|^{\kappa}_{q}+\|u^-\|^{\kappa}_{q}\Big].
\end{aligned}
\]
The  last equality is due to the fact that $\tau_\lambda u^+$ and $\tau_\lambda u^-$ have disjoint supports while
in the last inequality we exploited 
the GNS inequality. Let us set $\int (u^+)^q=t$ and $\int (u^-)^q=1-t$. We can suppose $t\in(0,1)$, since $u$
changes sign. Then the previous formula is equivalent to
\[
f(t)=\big(t^{\kappa/q}+(1-t)^{\kappa/q}\big)^{1/\kappa}-1\leq\delta(u).
\]
The function  $f$ is symmetric in $[0,1]$, vanishes at $0$ and $1$ and since $\kappa<q$ 
(by Lemma \ref{FeGcollegate}) is concave. Thus there exists $a>0$ such that
\[
f(t)\geq \frac{1}{a}\min\{t,1-t\},
\]
\noindent so that 
\begin{equation}\label{min}
\min\Bigg\{\int (u^+)^q, \int(u^-)^q\Bigg\}\leq a\delta(u).
\end{equation}
Suppose that the minimum in \eqref{min} is achieved by $\int(u^-)^q$ (being analogous the other case). 
Since $\delta(|u|)=\delta(u)$, we can conclude, thanks to the triangle inequality and to \eqref{min}, that
\[
\lambda(u)^{1/q}\leq\lambda(|u|)^{1/q}+\Bigg(\int|u-|u||^q\Bigg)^{1/q}\leq C_8(\delta(u)^{\xi/q}+\delta(u)^{1/q})
\le \kappa_1\delta(u)^{\a/q}
\]
where $\alpha=\min\{\xi, 1\}$ and $C_8$ a positive constant not depending on $u$. The last inequality holds for
$\delta(u)<1$. So  \eqref{formulagenerale1} holds with $\a_1=\a/q$. 
\end{proof}

\section{Reduction inequalities}\label{reduc}
The goal of this section is to prove the reduction inequalities \eqref{riduzione}. Namely we shall prove the 
following result.

\begin{theorem}\label{t2}
Assume the hypotheses of Theorem \ref{t1}. Then there exist two positive constants $\kappa_2$ and $\alpha_2$ such
that for every non-negative function $u\in D^{p,s}(\R^n)$ there exists an $n$-symmetric function $\bar{u}$ such
that the following reduction inequalities hold true
\begin{equation}\label{riduzioneteo}
\lambda(u)\leq \kappa_2\lambda(\bar{u})^{\alpha_2},\qquad   \delta(\bar{u})\leq \kappa_2\delta(u)^{\alpha_2}.
\end{equation}
\end{theorem}

The proof is accomplished arguing similarly to the Sobolev case considered in \cite{CFMP}, although some technical
modifications 
are needed. We begin recalling that if $v$ is an optimal function for \eqref{GNS}, then any other optimal
function is of the form
\[
v_{a,b,x_0}(x)=av(b(x-x_0)),
\]
where $a$ and $b$ are non-null constants and $x_0\in\R^n$. We define the {\em relative asymmetry} of a function on 
an affine subspace $S$ of $\R^n$ as
\[
\lambda(u|S)=\inf_{(a,b,x_0)\in\R^2\times S}\Big\{\frac{\|u-v_{a,b,x_0}\|_q^q}{\|u\|_q^q}:\text{v optimal for 
\eqref{GNS}}, \|v_{a,b,x_0}\|_q=\|u\|_q \Big\}.
\]
\noindent The next Lemma shows that the infimum in the definition of the relative asymmetry 
is achieved.
\begin{lemma}
Let $u\in D^{p,s}(\R^n)$ and $S$ an affine space contained in $\R^n$. Then  the infimum in the definition of 
$\lambda(u)$ and $\lambda(u|S)$ are achieved.
\end{lemma}
\begin{proof}
Since the two cases are analogous, we show a proof just for the asymmetry. We can suppose without loss of
generality 
that $\|u\|_q=1$. Let us start observing that $\lambda(u)<2$. Indeed, if $v$ is a competitor in the definition of
$\lambda(u)$, then, up to translate the center of symmetry of $v$, we can suppose that $u$ and $v$ do not have
disjoint supports. Then
\begin{equation}
\lambda(u)\leq\int|v-u|^q=\int_{\{u>v\}}(u-v)^q+\int_{\{v>u\}}(v-u)^q<\int u^q+\int v^q=2.
\end{equation}
\noindent Let now $v_h(x)=a_hv(b_h(x-x_h))$ be a sequence of functions such that $\|v_h\|_q=1$ and
$\|u-v_h\|_q^q\to\lambda(u)$ as $h\to\infty$. We want to show that, up to subsequences,
$(a_h,b_h,x_h)\to(a,b,x_0)\subset\R^2\times\R^n$ as $h\to\infty$. We have that
\begin{equation}
1=\int|v_h|^q=a_h^q\int|v(b_h(x-x_h))|^q=\frac{a_h^q}{b_h^n}\int|v|^q=\frac{a_h^q}{b_h^n}
\end{equation}
\noindent that is $a_h^q=b_h^n$. Since $v\in L^q(\R^n)$ there exists a function $\rho(\varepsilon)$ converging to
$0$ as 
$\varepsilon\to0$ such that for each $z\in\R^n$ we have
\begin{equation}\label{aaa}
\int_{B(z,\varepsilon)}|v|^q\leq\rho(\varepsilon),\qquad\int_{B(0,1/\varepsilon)}|v|^q\geq1-\rho(\varepsilon).
\end{equation}
\noindent Set now $b_-=\liminf_{h\to\infty}b_h$ and $b_+=\limsup_{h\to\infty}b_h$. We claim that $b_->0$ and
that $b_+<\infty$. Suppose $b_-=0$; then, recalling that $v_h$ is a radial function, we have
\begin{equation}
\begin{aligned}
\int_{B(0,1/\varepsilon)}|v_h|^q&\leq\int_{B(x_h,1/\varepsilon)}|v_h|^q=\frac{a_h^q}{b_h^n}
\int_{B(0,b_h/\varepsilon)}|v(y)|^qdy\\
&=\int_{B(0,b_h/\varepsilon)}|v(y)|^qdy,
\end{aligned}
\end{equation}
\noindent and the last quantity, up to pass to a subsequence, converges to $0$ as $h\to\infty$. Thus we can
suppose 
that, for a fixed $\varepsilon$ and for  $h$ big enough, we have
\begin{equation}\label{aaaaa}
\int_{B(0,1/\varepsilon)}|v_h|^q\leq\varepsilon.
\end{equation}
\noindent So, thanks to \eqref{aaaaa} and \eqref{aaa} we have
\begin{equation*}
\begin{aligned}
\|u-v_h\|_q^q&=\int_{B(0,1/\varepsilon)}|u-v_h|^q+\int_{B(0,1/\varepsilon)^c}|u-v_h|^q \\
		&\geq\left|\Bigg(\int_{B(0,1/\varepsilon)}|u|^q\Bigg)^{1/q}-\Bigg(\int_{B(0,1/\varepsilon)}|v_h|^q\Bigg)^{1/q}\right|^q \\
	& +\Bigg|\Bigg(\int_{B(0,1/\varepsilon)^c}|v_h|^q\Bigg)^{1/q} - \Bigg(\int_{B(0,1/\varepsilon)^c}|u|^q\Bigg)^{1/q}\Bigg|^q \\
		&\geq[ (1-\rho(\varepsilon))^{1/q}-\varepsilon^{1/q}]^q+[(1-\varepsilon)^{1/q}-\rho(\varepsilon)^{1/q}]^q.
\end{aligned}
\end{equation*}
\noindent Passing to the limit in $h\to\infty$ and then in $\varepsilon\to0$ we obtain that $\lambda(u)\geq2$, that
is a 
contradiction. Suppose now that $b_+=\infty$. Then 
\begin{equation}
\begin{aligned}
\int_{B(x_h,\varepsilon)^c}|v_h|^q=&\int_{B(x_h,\varepsilon)^c}a_h^q|v(b_h(x-x_h))|^qdx =a_h^q
\int_{B(0,\varepsilon)^c}|v(b_hx)|^qdx\\
	&=\frac{a_h^q}{b_h^n}\int_{B(0,\varepsilon)^c}|v(z)|^qdz=\int_{B(0,b_h\varepsilon)^c}|v|^q
\end{aligned}
\end{equation}
\noindent and arguing as before we can suppose that $\int_{B(x_h,\varepsilon)^c}|v_h|^q\leq\varepsilon$
\noindent for $h$ big enough. By  \eqref{aaa} we get
\[
\begin{aligned}
\|u-v_h\|_q^q &=\int_{B(x_h,\varepsilon)}|u-v_h|^q+\int_{B(x_h,\varepsilon)^c}|u-v_h|^q \\
		&\geq\left|\left(\int_{B(x_h,\varepsilon)}|v_h|^q\right)^{1/q}-\left(\int_{B(x_h,\varepsilon)}|u|^q\right)^{1/q}\right|^q\\
& +\left|\left(\int_{B(x_h,\varepsilon)^c}|u|^q\right)^{1/q}-\left(\int_{B(x_h,\varepsilon)^c} |v_h|^q\right)^{1/q}\right|^q \\
		&\geq(1-\varepsilon)-\rho(\varepsilon)+(1-\rho(\varepsilon))-\varepsilon
\end{aligned}
\]
\noindent and we get, as above, a contradiction.
\noindent Suppose now that $(x_h)_h$ is not bounded and extract a subsequence (not relabelled) such that
$|x_h|\to\infty$. Then given $N>0$, if $h$ is big enough we would get $\int_{B(x_h,N)}|u|^q\le1/N$.
\noindent If we choose $N$ such that $\int_{B(x_h,N)}|v_h|\geq\varepsilon$ for all $h\in\N$, we obtain
\[
\begin{aligned}
\|u-v_h\|_q^q&\geq\int_{B(x_h,N)} |u-v_h|^q+\int_{B(x_h,N)^c}|u-v_h|^q \\
		&\geq\Bigg|\Bigg(\int_{B(x_h,N)}|v_h|^q\Bigg)^{1/q}-\Bigg(\int_{B(x_h,N)}|u|^q\Bigg)^{1/q}\Bigg|^q \\
	&+\Bigg|\Bigg(\int_{B(x_h,N)^c}|u|^q\Bigg)^{1/q}-\Bigg(\int_{B(x_h,N)}|v_h|^q\Bigg)^{1/q}\Bigg|^q \\
		&\geq \left[(1-\varepsilon)^{1/q}-\frac{1}{N^{1/q}}\right]^q+\left[\left(1-\frac{1}{N}\right)^{1/q}-\varepsilon^{1/q}\right]^q.\end{aligned}
\]
and again we get a contradiction.
\end{proof}
  Clearly the asymmetry of a function estimates from below its relative asymmetry. But if we consider a
$n$-symmetric function, also the opposite estimate is true, as shown in next lemma.
\begin{lemma}\label{asrel}
Let $u\in D^{p,s}(\R^n)$ $k$-symmetric with respect to $k$ orthogonal hyperplanes and let $S$ be the intersection of such hyperplanes. Then
\[
\lambda(u|S)\leq3^q\lambda(u).
\]
\end{lemma}
\begin{proof}
Suppose as usual that $\|u\|_q=1$. Let $v_{a,b,x}$ a minimum for $\lambda(u)$. We consider now the orthogonal
projection $x_S$ of $x$ on $S$ and $y$ the symmetric point of $x_S$ with respect to $S$. 
Notice that since $u$ is symmetric with respect to $S$, also $v_{a,b,y}$ is a minimum for $\lambda(u)$. Moreover,
since the minima of the asymmetry are radial symmetric functions with decreasing profile, we have
\[
\|v_{a,b,x}-v_{a,b,x_S}\|_q\leq\|v_{a,b,x}-v_{a,b,y}\|_q.
\]
This observation and the triangle inequality imply that
\[
\begin{aligned}
\lambda(u)^{1/q}&\leq\|u-v_{a,b,x_S}\|_q\leq \|u-v_{a,b,x}\|_q+\|v_{a,b,x}-v_{a,b,x_S}\|_q \\
		&=\lambda(u)^{1/q}+\|v_{a,b,x}-v_{a,b,x_S}\|_q\leq\lambda(u)^{(1/q)}+\|v_{a,b,x}-v_{a,b,y}\|_q\\
		&\leq \lambda(u)^{1/q}+\|v_{a,b,x}-u\|_q^q+\|u-v_{a,b,y}\|_q=3\lambda(u)^{1/q}
\end{aligned}
\]
and the conclusion follows.
\end{proof}
\noindent The next result shows that the $n$-symmetry condition in the previous lemma can in some sense be relaxed.
\begin{lemma}\label{asrel2}
Consider a function $u\in L^q(\R^n)$, $u\ge0$, $H$ an 
hyperplane of $\R^n$ and $H^+$ and $H^-$ the half spaces having $H$ as boundary. Suppose that 
\[
\int_{H^+}|u|^q=\int_{H^-}|u|^q=\frac{1}{2}\int_H|u|^q,
\]
then
\begin{equation}\label{B}
\lambda(u|H)\leq C_0\lambda(u)^{1/q},
\end{equation}
with a constant $C_0$ depending only on $q$ and $n$.  
Moreover, if $T_H$ denotes the reflection with respect to $H$ of $\R^n$, it holds
\begin{equation}\label{b}
\int|u\circ T_H-u|^q\leq C_0\|u\|_q^q\lambda(u)^{1/q}.
\end{equation}
\end{lemma}
\begin{proof}
Suppose without loss of generality that $\|u\|_q=1$ and let $v_0=v_{a,b,x_0}$ a minimum for $\lambda(u)$ centred
at $x_0$. 
Suppose moreover that $x_0\in H^+$, being the other case analogous, and let $\bar{x}$ the projection of $x_0$ on
$H$ and $\bar{v}=v_{a,b,\bar{x}}$. Then
\begin{equation}\label{bb}
\lambda(u|H)\leq\int|u-\bar{u}|^q\leq2^{q-1}\Bigg(\lambda(u)+\int|v_0-\bar{v}|^q\Bigg).
\end{equation}
Consider the translated half spaces $K^{\pm}=H^\pm+(x_0-\bar{x})$. Since $x_0\in H^+$ it follows that $K^+\subseteq H^+$ and $H^-\subseteq K^-$. We have that
\[
\frac{1}{2}=\int_{K^\pm}v_0^q=\int_{H^\pm}u^q=\int_{H^\pm}\bar{v}^q
\]
and 
\[
\int_{H^-}|v_0-\bar{v}|^q=\int_{K^+}|v_0-\bar{v}|^q\leq\int_{H^+}|v_0-\bar{v}|^q;
\]
hence
\begin{equation}\label{bbb}
\int|v_0-\bar{v}|^q\leq2\int_{H^+}|v_0-\bar{v}|^q.
\end{equation}
Since $v_0\geq\bar{v}$ on $K^+$ we get that $|v_0-\bar{v}|^q\leq v_0^q-\bar{v}^q$ on $K^+$. Then
\begin{equation}\label{bbbb}
\begin{aligned}
\int_{K^+}|v_0-\bar{v}|^q &\leq\int_{K^+}v_0^q-\int_{K^+}\bar{v}^q=\frac{1}{2}-\int_{H^-}v_0^q \\
	&\leq C\Big(\|u\|_{L^q(H^-)}-\|v_0\|_{L^q(H^-)}\Big)\leq C\|u-v_0\|_q=C\lambda(u)^{1/q},
\end{aligned}
\end{equation}
for a suitable positive constant $C$. Moreover
\[
\begin{aligned}
\int_{H^+\setminus K^+}|v_0-\bar{v}|^q&\leq2^{q-1}\int_{H^+\setminus K^+}(v_0^q+\bar{v}^q)=2^q\int_{H^+\setminus K^+}v_0^q \\
	&=2^q\Bigg[\int_{H^+}v_0^q-\frac{1}{2}\Bigg]=2^q\Bigg[\int_{H^+}v_0^q-\int_{H^+}u^q\Bigg]
\end{aligned}
\]
and reasoning as in \eqref{bbbb} we obtain
\begin{equation}\label{bbbbb}
\int_{H^+\setminus K^+}|v-\bar{v}|^q\leq C\lambda(u)^{1/q}.
\end{equation}
Inequality \eqref{B} is then a consequence of \eqref{bb}, \eqref{bbb}, \eqref{bbbb}, and \eqref{bbbbb}.
We are left to show inequality \eqref{b}. Let $\hat{u}$ be the optimal function in the definition of
$\lambda(u|H)$. Then
\[
\begin{aligned}
\int_{H^\pm}|u\circ T_H-u|^q&\leq2^{q-1}\Bigg(\int_{H^\pm}|u\circ T_H-\hat{u}|^q+\int_{H^\pm}|u-\hat{u}|^q\Bigg) \\
	&=2^{q-1}\int|u-\hat{u}|^q=2^{q-1}\lambda(u|H)\leq C_0\lambda(u)^{1/q}.
\end{aligned}
\]
\end{proof}

Before passing to the proof of Theorem \ref{t2} we need another technical lemma which, roughly speaking, states
that if two optimal functions for the GNS inequality are near in $L^q$ norm, then their $L^q$ distance on the whole
$\R^n$ can be estimated from above  by their $L^q$ distance on just a {\em quarter} of $\R^n$. Its proof is quite
technical but it is
essentially based on a Taylor expansion.

\begin{lemma}\label{quarto}
Let $u$ be an optimal function for the GNS  inequality of parameters $s,q,p$ centred in $0$ with $\|u\|_q=1$, 
and set $u_{\alpha,z}(x)=\alpha^{n/q}u(\alpha(x-z))$ (for simplicity, $u_{1,z}=u_z$). Consider two orthogonal half
spaces $H$ and $K$ containing the origin on their boundaries. There exist two constants $K=K(n,s,q,p)>0$ and
$\tilde{\rho}<<1$ such that if
\[
\int|u_{\lambda,x_0}-u_{\mu,y_0}|^q\le\tilde{\rho},
\] 
then
\begin{equation}\label{th}
\int_{H\cap K}|u_{\lambda,x_0}-u_{\mu,y_0}|^q\ge K\int|u_{\lambda,x_0}-u_{\mu,y_0}|^q.
\end{equation}

\end{lemma}

\begin{proof}
Up to a rotation we can consider $H=\{e_1=0\}$ and $K=\{e_2=0\}$. So that $Q:=H\cap K=\{e_1\ge0,e_2\ge0\}$.
Let us define
\[
T_k=\{x: 1/k<|\nabla u(x)|< M\} \qquad M=\max_{\R^n}|\nabla u(x)|.
\]
Since $u$ is a radial function we have that each $T_k$ is a radial set composed of a countable union of centred
annuli, i.e. there exists a non-decreasing sequence of positive numbers $(r_{k,j})_j$ such that
\[
T_k=\bigcup_{j\in\N}\left(B_{r_{k,j+1}}\setminus B_{r_{k,j}}\right) 
\] 
where $B_r$ is the ball centred at the origin of radius $r$.
Let now $I_k=\{j\in\N:r_{k,j+1}-r_{k,j}\ge1/k\}$ and set
\[
S_k=\bigcup_{j\in I_k}\left( B_{r_{k,j+1}}\setminus B_{r_{k,j}}\right).
\]
We consider, for $z\in\R^n$, the biggest centred annulus contained in $S_k\cap(S_k+z)$:
\[
\Sigma(k,z)=\left\{x\in S_k\cap(S_k+z):\partial B_{|x|}\subseteq S_k\cap(S_k+z)\right\}.
\]
Notice, that $\Sigma(k,z_1)=\Sigma(k,z_2)$ whenever $|z_1|=|z_2|$ so we may define $\Sigma(k,r)=\Sigma(k,z)$ if $|z|=r$.
\noindent
Clearly $\bigcup_{k\in\N} S_k=\R^n$. Moreover, since $|\nabla u|$ is continuous, we have that 
\begin{equation}\label{=}
\bigcup_{|z|>0}\Sigma(k,z)=S_k.
\end{equation}
Indeed, if $x\in S_k$, there exists $r=r(x)>0$ such that $|\nabla u|(y)> 1/k$ for every $y\in B(x,r(x))$, that is
$\overline{B(x,r(x))}\subseteq S_k$. 
Thus it is sufficient to choose $|z|<\mathrm{dist}(x,\partial S_k)$ and we get that $\partial B_{|x|}\subseteq
S_k\cap (S_k+z)$, i.e $x\in \Sigma(k,z)$.  
Bearing these definitions in mind, we pass to prove inequality \eqref{th}. Up to a change of variables we
can consider just the case $\lambda=1+l>1$, $\mu=1$. 
\noindent
For any Borel set $A$ we have:

\begin{equation}
\begin{aligned}
&\int _A |\lambda^{n/q}u(\lambda(x-x_0))-u(x-y_0)|^q= \int_{A+y_0} |(1+l)^{n/q}u((1+l)(x+y_0-x_0))-u(x)|^q\\
	&=\int_{A+y_0}\Big|(1+l)^{n/q}\left[ u(x)+\big\langle \nabla
u(x),lx+(1+l)(y_0-x_0)\big\rangle\right]-u(x)+R(x)\Big|^q\\
&=\int_{A+y_0}\Big|\langle \nabla u(x),lx+(1+l)(y_0-x_0)\rangle+\frac nq l\langle
\nabla u(x),lx+(1+l)(y_0-x_0)\rangle+\frac nq lu(x)+R(x)\Big|^q\\
&=\int_{A+y_0} \left||\nabla u(x)|	\left\langle \frac{x}{|x|},y_0-x_0\right\rangle+ l\left[|x||\nabla
u(x)|+\frac nq u(x)\right]+R(x)\right |^q:=\int_{A+y_0}\left| E_{x_0,y_0,l}(x)+R(x)\right|^q
\end{aligned}
\end{equation}
where the last inequality is due to the radial symmetry of $\nabla u$ and the error term $R(x)$ is given by
\[
R(x)=(1+l)^{n/q}\left[\frac{|\nabla ^2u(x)|}{|x|^2}(x\otimes
x)(lx+(1+l)(y_0-x_0))^2\right]=O(l^2)+O(|lx+(1+l)(y_0-x_0)|^2).
\]
Notice that there exists $\tilde\rho_1$ such that for every $x\in S_k$
\begin{equation}
|R(x)|\le\frac12| E_{x_0,y_0,l}(x)|
\end{equation}
if $|x_0|+|y_0|+l\le \tilde\rho_1$, since $R$ is an infinitesimal of higher order than $E_{x_0,y_0,l}$.
\noindent We aim now to find $\tilde\rho$ such that, for $|x_0|+|y_0|+l\le\tilde\rho$, the following chain of 
inequalities holds true:
\begin{equation}\label{chain}
\int \left|E_{x_0,y_0,l}+R\right|^q\le c_1 \int_{\Sigma_{k,y_0}}|E_{x_0,y_0,l}+R|^q\le c_2\int_{\Sigma_{k,y_0}\cap Q}|E_{x_0,y_0,l}+R|^q\le c_2\int_Q|E_{x_0,y_0,l}+R|^q
\end{equation}
for suitable constants $c_1$, $c_2$ and $k$. This would immediately imply inequality \eqref{th}. We begin remarking
that 
\[
\lim_{k\to\infty}\int_{S_k}|E_{x_0,y_0,l}+R|^q=\int| E_{x_0,y_0,l}+R|^q
\]
so there exists $\overline k$ such that for $k\ge \overline k$ we have
\[
\int_{S_k}|E_{x_0,y_0,l}+R|^q\ge\frac12\int|E_{x_0,y_0,l}+R|^q.
\]
Moreover, in view of \eqref{=} we get, 
 for any $|y_0|\le\tilde\rho$ small enough,
\[
\int_{\Sigma(k,y_0)}|E_{x_0,y_0,l}+R|^q\ge\frac12\int_{S_k}| E_{x_0,y_0,l}+R|^q.
\]
So the first inequality in \eqref{chain} holds with $c_1=1/4$. Since the last inequality in \eqref{chain} is
trivial, 
we are left to prove the central one. Suppose by contradiction that there is a sequence
$(x_h,y_h,l_h)\to0\in\R^n\times\R^n\times\R$ as $h\to\infty$ such that \eqref{chain} does not hold.
\noindent
Reasoning as before we get
\begin{equation}
\begin{aligned}
&\int_{\Sigma(k,y_h)}|E_{x_h,y_h,l_h}+R|^q\le(3/2)^q\int_{\Sigma(k,y_h)}|E_{x_h,y_h,l_h}|^q\\
&\le(3/2)^q \int_{\Sigma(k,y_h)}|M\left\langle \frac{x}{|x|},y_h-x_h\right\rangle +l_h(|x||\nabla u(x)|+(n/q)
u(x))|^q\\
&\le(3/2)^q\int_{\Sigma(k,y_h)}|M|y_h-x_h| +l_h(|x||\nabla u(x)|+(n/q) u(x))|^q \\
&\le(3/2)^q\int_{\Sigma(k,y_h)}|M|v_h| +c_0l_h|^q
\end{aligned}
\end{equation}
where we set $v_h=y_h-x_h$ and
\[
c_0=\limsup_{h\to\infty}\max_{\Sigma(k,y_h)}|x||\nabla u(x)|+\frac nqu(x)>0.
\]
So there exist two constants $k_0,k_1$ such that
\begin{equation}\label{c2}
\int_{\Sigma(k,y_h)}|E_{x_h,y_h,l_h}+R|^q\le k_0\int_{\Sigma(k,y_h)}|k_1l+|v_h||^q.
\end{equation}
On the other hand we have 
\[
\begin{aligned}
&\int_{\Sigma(k,y_h)\cap Q}|E_{x_h,y_h,l_h}+R|^q\ge\frac{1}{2^q}\int_{\Sigma(k,y_h)\cap Q}|E_{x_h,y_h,l_h}|^q\\
&=\frac{1}{2^q}\int_{\Sigma(k,y_h)}\left| |\nabla u(x)|\left\langle
\frac{x}{|x|},\frac{v_h}{|v_h|}\right\rangle|v_h| +l(|x||\nabla u(x)|+(n/q) u(x))  \right|^q.
\end{aligned}
\]
We have now three possible situations: $|v_h|<<l_h$, $l_h<<|v_h|$ or $l_h\simeq|v_h|$ as $h\to\infty$. In the first
case we have, thanks to \eqref{c2}, that
\[
\int_{\Sigma(k,y_h)}|E_{x_h,y_h,l_h}+R|^q\le k_2\int_{\Sigma(k,y_h)}l_h^q
\]
for a suitable $k_2$. Moreover it is easy to find positive constants $k_3$ and $k_4$ not depending on $l_h$ and
$v_h$ such that
\[
\int_{\Sigma(k,y_h)\cap Q}|E_{x_h,y_h,l_h}+R|^q\ge k_3\int_{\Sigma(k,y_h)\cap Q}l_h^q\ge k_4
\int_{\Sigma(k,y_h)}l_h^q.
\]
So in this case \eqref{chain} holds with $c_2=\frac{k_3}{k_4}c_1$ (or, in other terms, we get a contradiction). The
second case, $l_h<<|v_h|$ can be solved with the same argument, with the only observation that, for the estimate
from above, we must  further restrict the set of integration to the set (for instance) $U=\{x\in Q:
|\langle x/|x|,v_0/|v_0|\rangle|\ge1/10 \}$. We are left to study the case where $|v_h|\simeq l_h$. If $\liminf_h
|v_h|/l_h\ge\widetilde c>0$, we have that
\[
\int_{\Sigma(k,y_h)}|E_{x_h,y_h,l_h}+R|^q=l_h^q\int_{\Sigma(k,y_h)}\left|c_0\frac{l_h}{|v_h|}+M\right|^q.
\]
Let us define $V=\{x\in Q: |\langle x/|x|,v_h/|v_h|\rangle|\le\alpha \}$ where $\alpha$ is a constant 
(depending on $\widetilde c$) that will be fixed later. We have
\[
\begin{aligned}
&\int_{\Sigma(k,y_h)\cap Q}|E_{x_h,y_h,l_h}+R|^q\ge\int_{\Sigma(k,y_h)\cap V}|E_{x_h,y_h,l_h}+R|^q\\
&\ge \kappa_0 l^q\int_{\Sigma(k,y_h)\cap V}\left| c_1\frac{l}{|v_h|} -M\alpha \right|^q.
\end{aligned}
\]
Choosing $\alpha$ small enough, since $l_h/|v_h|>>0$, we can find constants $k_1$ and $k_2$ not depending on 
$l_h$ and $v_h$ such
that 
\[
\int_{\Sigma(k,y_h)\cap Q}|E_{x_h,y_h,l_h}+R|^q\ge \kappa_1l^q\int_{\Sigma(k,y_h)\cap V}\left|c_0\frac{l_h}{|v_h|}
+M\right|^q\ge\kappa_2l^q\int_{\Sigma(k,y_h)}\left|c_0\frac{l_h}{|v_h|}+M\right|^q,
\]
so again \eqref{chain} holds with $c_2=k_2/c_1$ (thus again a contradiction). If 
$\liminf_h l_h/|v_h|\ge\widetilde c>0$, we can conclude applying an analogous argument. 
\end{proof}

We pass now to prove Theorem \ref{t2}. For the sake of clearness we  divide its proof into two parts. We first
prove a 
proposition which provides us a method to pass from a generic function in $D^{p,s}(\R^n)$ to an $(n-1)$-symmetric
function which satisfies  the reduction inequalities \eqref{riduzioneteo}. Then we shall see how to obtain the last
required symmetry.
\begin{proposition}\label{n-1}
There exists a positive constant $C$ such that for every function $u\in D^{p,s}(\R^n)$ there is an
$(n-1)$-symmetric 
function $\tilde{u}$ such that
\begin{equation}\label{riduzione2}
\lambda(u)\leq C\lambda(\tilde{u}),\qquad \delta(\tilde{u})\leq2^{n-1}\delta(u).
\end{equation}
\end{proposition}
\begin{proof}
As usual, by the homogeneity of the deficit and the asymmetry, we can consider $\|u\|_q=1$. Moreover we can suppose
that $\delta(u)<\bar{\delta}$ for an arbitrary small $\bar{\delta}$. Indeed, if $\delta(u)\geq\bar{\delta}$, let
$v$ be a radial (and so $n$-symmetric!) function such that $0<\delta(v)<2^{n-1}\bar{\delta}$. Then
\[
\lambda(u)\leq2^q=\frac{2^q}{\lambda(v)}\lambda(v)\leq \bar{C}\lambda(v),\qquad \delta(v)\leq2^{n-1}
\bar{\delta}\leq2^{n-1}\delta(u).
\]
\noindent Consider, for $k=1\dots n$, the $n$ hyperplanes orthogonal to the coordinate axis such that
\[
\int_{H_k^+}u^q=\int_{H_k^-}u^q=\frac{1}{2}
\]
where $H_k^\pm$ are the two half spaces in which $\R^n$ is divided by $H_k$. Denoting by $T_k$ the reflection with
respect to $H_k$, we define
\begin{equation}
\begin{aligned}
&u_k^+(x)=
\left\{
\begin{array}{ccc}
	u(x)&\text{if}&x\in H^+_k \\
	u(T_k(x))&\text{if}& x\in H^-_k
\end{array}
\right.\\
&\text{ }\\
&u_k^-(x)=
\left\{
\begin{array}{ccc}
u(x) &\text{if}& x\in H^-_k \\
u(T_k(x))&\text{if}& x\in H^+_k
\end{array}
\right.
\end{aligned}
\end{equation}
\noindent By construction $u_k^\pm$ are symmetric with respect to $H_k$. We observe now that
\[
\int u^s=\frac{\int (u_k^+)^s+\int(u_k^-)^s}{2}\text{,}\qquad \int |\nabla
u|^p=\frac{\int|\nabla u_k^+|^p+\int|\nabla u_k^-|^p}{2},
\]
and since $t\to t^{1/p}$ and $t\to t^{1/s}$ are concave functions, we have that
\begin{equation}\label{d}
\|u\|_s\geq\frac{\|u_k^+\|_s+\|u_k^-\|_s}{2}\text{,}\qquad \|\nabla u\|_p\geq\frac{\|\nabla
u_k^+\|_p+\|\nabla u_k^-\|_p}{2}.
\end{equation}
\noindent By the definition of $\delta(u)$ and since $(x,y)\mapsto x^\theta y^{1-\theta}$ is concave on $\R_+^2$ and strictly increasing in $x$ and $y$, we get that
\[
\begin{aligned}
G\delta(u)&\geq\Bigg(\frac{\|\nabla
u_k^+\|_p+\|\nabla u_k^-\|_p}{2}\Bigg)^\theta\Bigg(\frac{\|u_k^+\|_s+\|u_k^-\|_s}{2}\Bigg)^{1-\theta}-G\\
	&\geq\frac{G}{2}\delta(u_k^+)+\frac{G}{2}\delta(u_k^-),
\end{aligned}
\]
and so 
\[
\delta(u)\geq\frac{\delta(u_k^+)+\delta(u_k^-)}{2}.
\]
\noindent In particular for every $k=1,\dots,n$
\begin{equation*}
\max\{\delta(u_k^+),\delta(u_k^-)\}\leq2\delta(u).
\end{equation*}
\noindent Let $v_k^+$ and $v_k^-$ be the functions which minimize $\lambda(u_k^\pm|H_k)$. Then, by triangle inequality and Lemma \ref{asrel} we have 
\[
\begin{aligned}\label{dd}
\lambda(u)&\leq\int|u-v_k^+|^q=\int_{H_k^+}|u_k^+-v_k^+|^q+\int_{H_k^-}|u_k^--v_k^+|^q \\
	&\leq2^{q-1}\Bigg(\frac{\lambda(u_k^+|H_k)+\lambda(u_k^-|H_k)}{2}+\int_{H_k^-}|v_k^+-v_k^-|^q\Bigg) \\
	&\leq2^{q-2}3^q\Bigg(\lambda(u_k^+)+\lambda(u_k^-)+\int_{H_k^-}|v_k^+-v_k^-|^q\Bigg).
\end{aligned}
\]
\noindent
We claim that if $\delta(u)$ is small enough, then for any couple of indexes $1\leq i<j\leq n$, for $k_1=i$
or $k_1=j$ the following inequality holds:
\begin{equation}\label{claim1}
\int_{H_{k_1}^-}|v_{k_1}^+-v_{k_1}^-|^q\leq C\Bigg(\int_{H_{k_1}^+}|u_{k_1}^+-v_{k_1}^+|^q+
\int_{H_{k_1}^-}|u_{k_1}^--v_{k_1}^-|^q\Bigg).
\end{equation}
Let us show how this brings to the conclusion: by \eqref{d} and \eqref{claim1} we would have
\[
\lambda(u)\leq C\max\{\lambda(u_{k_1}^+),\text{ }\lambda(u_{k_1}^-)\}, \quad\max\{\delta(u_{k_1}^+),
\delta(u_{k_1}^-)\}\leq2\delta(u).
\]
Thus we would obtain that $u_{k_1}^+$ or $u_{k_1}^-$, say $u_{k_1}^+$, is a $1$-symmetric
function with $L^q$ norm equal to $1$. We can now iterate this procedure exploiting two
hyperplanes between the $(n-1)$ we left, obtaining a $2$-symmetric function which  satisfies the
reductions inequalities \eqref{riduzione2} (with $u_{k_1}^+$ in place of $u$). We can continue such a construction
until we have just one hyperplane left. But then we would have an $(n-1)$-symmetric function which satisfies
inequalities \eqref{riduzione2} that is the claim of the proposition.
Thus we are only left to prove \eqref{claim1}. We divide the proof of it into two further steps: 
\begin{itemize}
\item[\underline{Step 1}]There exists a positive constant $C_0$ such that for all $1\leq i<j\leq n$, $\sigma$, $\tau\in\{+,-\}$ it holds
\begin{equation}\label{ee}
\int|v_i^\sigma-v_j^\tau|^q\leq C_0 \int_{H_i^\sigma\cap H_j^\tau}|v_i^\sigma-v_j^\tau|^q.
\end{equation}
\item[\underline{Step 2}]  \eqref{ee} implies \eqref{claim1}.
\end{itemize}
To verify \eqref{ee} we notice that thanks to Lemma \ref{quarto} the inequality is verified if we are able to find
two positive constants $\rho$ and $C_1$ such that:
\begin{itemize}
\item[(i)] $\int (v_i^\sigma)^q=\int (v_j^\tau)^q=1$;
\item[(ii)] $H_i^\sigma$ ed $H_j^\tau$ are two orthogonal half spaces which contains on their boundary the 
center of symmetry of $v_i^\sigma$ e $v_j^\tau$;
\item[(iii)] $\int|v_i^\sigma-v_j^\tau|^q\leq\rho$.
\end{itemize}
\noindent We need to check $(i)-(iii)$. Clearly $(i)$ and $(ii)$ are true by construction. 
So we are left to verify $(iii)$. We have
\begin{equation}\label{p}
\|v_i^\sigma-v_j^\tau\|_q\leq\|v_i^\sigma-u_i^\sigma\|_q+\|u_i^\sigma-u\|_q+\|u-u_j^\tau\|_q+\|u_j^\tau-v_j^\tau\|_q.
\end{equation}
Thanks to Lemma \ref{asrel2}
\begin{equation}\label{pp}
\int|u_i^\sigma-u|^q=\frac{1}{2}\int|u\circ T_i-u|^q\leq C\lambda(u)^{1/q}.
\end{equation}
Moreover
\begin{equation}\label{ppp}
\int|v_i^\sigma-u_i^\sigma|^q\leq2^{q-1}\Big(\lambda(u|H_i)+\|u-u_i^\sigma\|^p_p\Big)\leq C\lambda(u)^{1/q}.
\end{equation}
The same estimate holds as well for the third and the fourth addend on the right hand side of \eqref{p}. Putting
together \eqref{pp} and \eqref{ppp} we obtain the claim of $(iii)$ and this conclude the proof of Step 1.
Let us prove now Step 2. Suppose, to fix the ideas, that $i=1$ and $j=2$. For $k=1,2$ we set
\[
h_k=v_k^+\chi_{H_k^+}+v_k^-\chi_{H_k^-}
\]
where $\chi_A$ denotes the characteristic function of the set $A$.
Thanks to \eqref{ee},
\begin{equation*}
\int|h_1-h_2|^q\geq\int_{H_1^+\cap H_2^+}|h_1-h_2|^q=\int_{H_1^+\cap
H_2^+}|v_1^+-v_2^+|^q\geq\frac{1}{C}\int|v_1^+-v_2^+|^q.
\end{equation*}
With a similar argument, using ${H_1^-\cap H_2^+}$ instead of $H_1^+\cap H_2^+$ we get
\[
\int|h_1-h_2|^q\geq\frac{1}{C}\int|v_1^--v_2^+|^q.
\]
Hence
\begin{equation}\label{f}
\int|v_1^+-v_1^-|^q\leq2^qC\int|h_1-h_2|^q.
\end{equation}
Similarly we can see that
\begin{equation}\label{ff}
\int|v_2^+-v_2^-|^q\leq2^qC\int|h_1-h_2|^q.
\end{equation}
Furthermore we have
\begin{equation}\label{fff}
\begin{aligned}
\int|h_1-h_2|^q&\leq2^{q-1}\Bigg(\int|h_1-u|^q+\int|h_2-u|^q\Bigg)\\
	&=2^{q-1}\Bigg(\int_{H_1^+}|v_1^+-u_1^+|^q+\int_{H_1^-}|v_1^--u_1^-|^q\\
	&+\int_{H_2^+}|v_2^+-u_2^+|^q+\int_{H_2^-}|v_2^--u_2^-|^q\Bigg) \\
	&\leq2^q\max\Bigg\{\int_{H_1^+}|v_1^+-u_1^+|^q+\int_{H_1^-}|v_1^--u_1^-|^q, \\
	&\int_{H_2^+}|v_2^+-u_2^+|^q+\int_{H_2^-}|v_2^--u_2^-|^q\Bigg\}\\
	&=2^q\Bigg(\int_{H_{k}^+}|v_{k}^+-u_{k}^+|^q+\int_{H_{k}^-}|v_{k}^--u_{k}^-|^q\Bigg).
\end{aligned}
\end{equation}
\noindent putting together \eqref{f}, \eqref{ff} and \eqref{fff} we obtain the claim of Step 2. 
\end{proof}
We stress that we can not symmetrize directly our function once again. A formal argument which shows a
problem 
that may occur is the following: consider a function $v$ such that $\delta(v)=\lambda(v)=0$ and construct a
function $u$ as follows:
\[
u(x)=v(x)\chi_{\{x_1\ge0\}}(x)+2^{n/q}v(2x)\chi_{\{x_1<0\}}(x).
\]
It follows that $u$ is $(n-1)$-symmetric in $H_k=\{x_k=0\}$ for $k\neq1$. if we try to symmetrize 
such function with respect to $H_1$ we would obtain $u_+(x)=2^{n/q}(v(2x))$ and $u_-(x)=v(x)$. Clearly none of them
 satisfy the first inequality in \eqref{riduzioneteo}. However we are going to see that a more refined kind of
symmetrization can bypass this problem. 
\begin{proof}[Proof of Theorem \ref{t2}]
Assume as usual that $\|u\|_q=1$. We can assume, thanks to Proposition \ref{n-1}, that $u$ is an $(n-1)$-symmetric 
function and that $\delta(u)<\bar{\delta}$ with $\bar{\delta}$ arbitrarily small. Up to a rotation and a
translation, we can consider $u$ to be symmetric with respect to the coordinate axes $\{x_k=0\}$ for $k=2,\dots,n$
and such that
\[
\int_{\{x_k>0\}}u^q=\frac{1}{2}=\int_{\{x_k<0\}}u^q.
\]
\noindent Let $u^\pm$ be the two symmetrizations of $u$ with respect to the hyperplane $\{x_1=0\}$, constructed 
as in Proposition \ref{n-1} . We have that $\max\{\delta(u^+),\delta(u^-)\}\leq2\delta(u)$. So, if
$\min\{\lambda(u^+),\lambda(u^-)\}\geq C_0\lambda(u)$, we would be done. Thus we can suppose that
\begin{equation}
\max\{\lambda(u^+),\lambda(u^-)\}<\varepsilon\lambda(u)
\end{equation}
for some constant $\varepsilon$ to be chosen. Consider $Q=\{|x_1|\leq x_2\}$, $Q^+=Q\cap\{x_1>0\}$ and 
$Q^-=Q\cap\{x_1<0\}$ and define a function $\hat{u}$ as follows:
\[
\hat{u}(x)=
\begin{cases}
u(x)\qquad&\text{if}\quad x\in Q\\
u(R_1x) &\text{if}\quad x\in R_1(Q)\\
\hat{u}(R_2x) &\text{if}\quad x\in R_2(Q\cup R_1(Q))
\end{cases}
\]
\noindent where $R_1$ and $R_2$ are the reflections with respect to $\{x_1=x_2\}$ and $\{x_1=-x_2\}$ respectively. 
The function $\hat{u}$ satisfies all the  symmetries of $u$ with the exception of the one related to the hyperplane
$\{x_2=0\}$, but by construction it is symmetric also with respect to $\{x_1=\pm x_2\}$. So it is $n$-symmetric. It
remains  to show that $\hat{u}$ satisfies the reduction inequalities \eqref{riduzioneteo}. Let us start with
the one which concerns the asymmetry. To this aim we denote by $\hat{v}$, $v^+$ and $v^-$ the functions which
achieve the minimum of $\lambda(\hat{u}|\{0\})$, $\lambda(u^+|\{0\})$ and $\lambda(u^-|\{0\})$ respectively. Since 
$\int \hat{u}^q\leq4$, we get that
\[
\begin{aligned}
3^q\lambda(\hat{u})&\geq\lambda(\hat{u}|\{0\})=\frac{\int|\hat{u}-\hat{v}|^q}{\int\hat{u}^q}=
\frac{4}{\int\hat{u}^q}\int_Q|u-\hat{v}|^q \\
	&\geq\int_{Q^+}|u^+-\hat{v}|^q+\int_{Q^-}|u^--\hat{v}|^q \\
	&=\int_{Q^+}|u^+-\hat{v}|^q+\int_{Q^+}|u^--\hat{v}|^q\geq\frac{1}{2^{q-1}}\int_{Q^+}|u^+-u^-|^q.
\end{aligned}
\]
\noindent The first inequality in \eqref{riduzioneteo} is then true if we can estimate $\lambda(u)$ in terms of 
$\int_{Q^+}|u^+-u^-|^q$. To this aim we observe that
\begin{equation}\label{g}
\begin{aligned}
\|u^+-u^-\|_{L^q(Q^+)}&=\frac{1}{2}\|u^+-u^-\|_{L^q(Q)}\\
	&=\frac{1}{2}\|(v^+-v^-)-(v^+-u^+)-(u^--v^-)\|_{L^q(Q)} \\
	&\geq\frac{1}{2}\Big[\|v^+-v^-\|_{L^q(Q)}-\|u^+-v^+\|_{L^q(Q)}-\|u^--v^-\|_{L^q(Q)}\Big].
\end{aligned}
\end{equation}
Moreover
\begin{equation}\label{gg}
\int_Q|u^\pm-v^\pm|^q\leq\int|u^\pm-v^\pm|^q=\lambda(u^\pm|\{0\})\leq3^q\lambda(u^\pm)\leq\varepsilon3^q\lambda(u)
\end{equation}
where we exploited Lemma \ref{asrel} and Lemma \ref{asrel2} and the fact that $u^\pm$ are $n$-symmetric functions.
\noindent Thanks to \eqref{g} and \eqref{gg} we obtain
\[
\|u^+-u^-\|_{L^q(Q^+)}\geq\frac{1}{2}\Big[\|v^+-v^-\|_{L^q(Q)}-2\Big(\frac{3^q}{C(\varepsilon)}\lambda(u)\Big)^{1/q}\Big],
\]
where $C(\varepsilon)$ is a suitable positive constant. Furthermore, always thanks to \eqref{gg} we have
\[
\begin{aligned}
\lambda(u)&\leq\int|u-v^+|^q=\frac{1}{2}\int|v^+-u^+|^q+\frac{1}{2}\int|v^+-u^-|^q\\
	&\leq\varepsilon\frac{3^q}{2}\lambda(u)+2^{q-2}\Bigg(\varepsilon3^q\lambda(u)+\int|v^+-v^-|^q\Bigg),
\end{aligned}
\]
\noindent and so
\begin{equation}
\int|v^+-v^-|^q\geq\frac{1}{2^{q-2}}\left[\lambda(u)-\varepsilon\frac{3^q}{2}\lambda(u)-\varepsilon3^q\lambda(u)
\right]
\geq\frac{1}{2^q}\lambda(u)
\end{equation}
where last inequality is justified if we choose $\varepsilon$ small enough. Summarizing we obtained
\[
\|u^+-u^-\|_{L^q(Q^+)}\geq C_0\lambda(u)^{1/q},
\]
where $C_0$ is a suitable positive constant not depending on $u$.
Let us consider now the second inequality in \eqref{riduzioneteo}, that is the one concerning the deficit. Notice
that we do not know if $\|\hat{u}\|_q=1$.
Nevertheless we have
\[
\Big|\|\hat{u}\|_{L^q(Q^+)}-(1/8)^{1/q}\Big|=\Big|\|u^+\|_{L^q(Q^+)}-\|v^+\|_{L^q(Q^+)}\Big|\leq\|u^+-v^+\|_{L^q(Q^+)}.
\]
and, since $|s^q-t^q|\leq C_1|s-t|$ for a suitable constant $C_1$ if $s$ and $t$ are in $[0,1]$, we get
\begin{equation}\label{j}
\begin{aligned}
\Bigg|\int_{Q^+}\hat{u}^q-\frac{1}{8}\Bigg|&\leq C_2\|u^+-v^+\|_q= C_2\lambda(u^+|\{0\})^{1/q} \\
	&\leq 3^{1/q}C_3\lambda(u^+)^{1/q}\leq 3^{1/q}C_4\delta(u^+)^{\alpha/q}\leq 2C_4\delta(u)^{\alpha/q},
\end{aligned}
\end{equation}
where we used Lemma \ref{asrel}, the fact that $u$ is $n$-symmetric and that $\delta(u^+)\le2\delta(u)$ and where 
$C_2$, $C_3$ and $C_4$ are positive constants. Analogous estimates hold on $Q^-$, on
$U^+=\{x_1>0,x_2>0\}\setminus Q$ and on $U^-=\{x_1<0,x_2>0\}\setminus Q$. Then, by the triangle inequality we get
\begin{equation}\label{40}
\Big|\int\hat{u}^q-1\Big|\leq C\delta(u)^{\alpha/q}
\end{equation}
for a suitable $\alpha>0$. Let us recall the definition of the functionals $F$ and $G$ given in \eqref{FeG1}:
\begin{equation}
G(u)=\|\nabla u\|_p^\theta\|u\|_s^{1-\theta},\qquad F(u)=\int|\nabla u|^p+\int|u|^s.
\end{equation}
By Lemma \ref{FeGcollegate} we know that there exist two positive constants $\kappa$ and $\eta_0$ such that for
every $u$ there exists $\lambda>0$ such that $F(\tau_\lambda u)=\eta_0 G(u)^\kappa$ where $\tau_\lambda
u(x)=\lambda^{n/q}u(\lambda x)$. 
Furthermore such $\lambda$ minimizes the function $\mu\mapsto F(\tau_\mu u)$ in $\R^+$. So we get
\begin{equation}\label{h}
G^\kappa \|\hat{u}\|_q^\kappa(1+\delta(\hat{u}))^\kappa =G(\hat{u})^\kappa
\leq\Bigg(\frac{1}{\eta_0}\int\big|\nabla \tau_\lambda
\hat{u}\big|^p+\frac{1}{\eta_0}\int\big|\tau_\lambda\hat{u}\big|^s\Bigg)
\end{equation} 
for all $\lambda>0$. Then
\begin{equation}\label{hh}
\begin{aligned}
&\frac{1}{\eta_0}\Bigg(\int\big|\nabla \tau_\lambda\hat{u}\big|^p+\int\big|\tau_\lambda\hat{u}\big|^s\Bigg)
=\frac{4}{\eta_0}\Bigg(\int_{Q^+\cup Q^-}\big|\nabla \tau_\lambda\hat{u}\big|^p+\int_{Q^+\cup
Q^-}\big|\tau_\lambda\hat{u}\big|^s\Bigg) \\
	&=\frac{4}{\eta_0}\Bigg(\int_{\{x_2>0\}}\big(\big|\nabla \tau_\lambda u\big|^p+\big|\tau_\lambda
u\big|^s\big)-
\int_{U^+\cup U^-}\big(\big|\nabla \tau_\lambda u \big|^p+\big|\tau_\lambda u\big|^s\big)\Bigg).
\end{aligned}
\end{equation}
\noindent Choosing $\lambda>0$ such that $F(\tau_\lambda u)=\eta_0 G(u)^\kappa$, since $u$ is symmetric with respect to $\{x_2=0\}$, we get
\begin{equation}\label{hhh}
\begin{aligned}
\int_{\{x_2>0\}}\big(\big|\nabla \tau_\lambda u\big|^p+\big|\tau_\lambda u\big|^s\big)&=\frac{1}{2}F(\tau_\lambda
u)\\
	&=\frac{\eta_0}{2}G(u)^\kappa=\frac{\eta_0G^\kappa}{2}(\delta(u)+1)^\kappa\\
	&\leq\frac{\eta_0G^\kappa}{2}(1+C_5\delta(u)),
\end{aligned}
\end{equation}
where the last inequality is true for $\delta(u)$ small enough. Let us consider now the function 
\[
v(x)=\begin{cases}
	\tau_\lambda u(x) & \quad x\in U^+\\
	\tau_\lambda u(R_1x)&  \quad x\in R_1U^+\\
	v(S_1x) &\quad x\in\{x_1<0,x_2>0\}\\
	v(S_2x) &\quad x\in\{x_2<0\}
\end{cases}
\]
where $S_i:\R^n\to\R^n$ denotes the symmetrization map with respect to $\{x_i=0\}$.
We have
\begin{equation}\label{hhhh}
\begin{aligned}
\int_{U^+}\big(\big|\nabla
\tau_\lambda(u)\big|^p+\big|\tau_\lambda(u)\big|^s\big)&=\frac{1}{8}\Bigg(\int\big(\big|\nabla
v\big|^p+\big|v\big|^s\big)\Bigg) \geq\frac{1}{8}\min_{\mu\in\R^+}F(\tau_\mu v)\\
	&=\frac{\eta_0}{8}G(v)^\kappa \geq\frac{\eta_0G^\kappa}{8}\Bigg(\int v^q\Bigg)^{\kappa/q}=\frac{\eta_0G^\kappa}{8}\cdot8^{\kappa/q}\Bigg(\int_{U^+}u^q\Bigg)^{\kappa/q}\\
	& \geq \eta_0G^\kappa 8^{-1+\kappa/q}(\frac{1}{8}-C_6\delta(u)^\alpha)^{\kappa/q} \geq
\eta_0G^\kappa(\frac{1}{8}-C_7\delta(u)^\alpha), 
\end{aligned}
\end{equation}
\noindent for suitable constants $C_6$ and $C_7$ and for $\delta(u)$ small enough. Notice that in the last
inequality we made use of  \eqref{j}. An analogous estimate can be obtained on $U^-$. By \eqref{h}, \eqref{hh},
\eqref{hhh} and \eqref{hhhh} we get
\begin{equation}\label{hhhhh}
\begin{aligned}
G^\kappa\|\hat{u}\|_q^\kappa(1+\delta(\hat{u}))^\kappa &\le\frac{4}{\eta_0}\Big[\frac{\eta_0G^\kappa}{2}(1+C_5\delta(u))-2\eta_0G^\kappa\left(\frac{1}{8}-C_7\delta(u)\right)^\beta\Big] \\
	&=G^\kappa\Big[1+C_8\delta(u)^\alpha\Big]
\end{aligned}
\end{equation}
for a suitable $C_8>0$ and where $\alpha$ is the minumum between  $1$ and $\beta$. By \eqref{hhhhh} and thanks to \eqref{40} we conclude that
\begin{equation*}
\delta(\hat{u})\leq\frac{1+C_9\delta(u)^\alpha}{1-C_9\delta(u)^\alpha}-1\leq C\delta(u)^\alpha
\end{equation*}
that is true again for $\delta(u)$ small enough. The conclusion follows setting $\kappa_2=C$.

\end{proof}

\section{Acknowledgements}
The author wish to warmly thank F. Maggi for its enlightening advise during the development of this work, and 
L. Brasco, G. De Philippis and A. Pratelli  for several useful discussions about  the topic.

\end{document}